\documentclass[regno, 12pt]{amsart}
\usepackage{graphicx,xcolor,subcaption} \usepackage{amsmath}
\usepackage{framed,comment,enumerate}
\usepackage{amssymb,amsmath,amsfonts,amsthm,esint,enumitem,comment,mathtools,centernot}

\usepackage{hyperref}

\usepackage[style=alphabetic]{biblatex}
\renewbibmacro*{in:}{%
  \ifentrytype{article}{}{\printtext{\bibstring{in}\intitlepunct}}}

\numberwithin{equation}{section}

\newtheorem{theorem}{Theorem}[section]
\newtheorem{corollary}[theorem]{Corollary}

\newtheorem{lemma}[theorem]{\textsc{Lemma}}
\newtheorem{proposition}[theorem]{Proposition}
\theoremstyle{definition}
\newtheorem{definition}[theorem]{Definition}

\newtheorem{remark}[theorem]{Remark}

\DeclareMathOperator{\Div}{div}
\DeclareMathOperator{\dist}{dist}
\DeclareMathOperator{\graph}{graph}
\DeclareMathOperator{\hyp}{hyp}
\DeclareMathOperator{\epi}{epi}

\DeclareMathOperator{\fl}{\textit{F}_\Lambda}

\DeclareMathOperator{\R}{\mathbb{R}}

\DeclareMathOperator{\Hh}{\mathcal{H}}

\newcommand{\eps}{\varepsilon}

\title[Boundary regularity for the one-phase problem]{Continuity up to the boundary for minimizers of the one-phase Bernoulli problem} 

\author{Xavier Fern\'andez-Real}
\address{Institute of Mathematics, \'Ecole Polytechnique F\'ed\'erale de Lausanne, Lausanne, Switzerland}
\email{xavier.fernandez-real@epfl.ch}

\author{Florian Gruen}
\address{Department of Mathematics, Graduate School of Science, Kyoto University, Kyoto, Japan}
\email{gruen.florian.32r@st.kyoto-u.ac.jp}



\thanks{X. F. was supported by the Swiss National Science Foundation (SNF grant  PZ00P2\_208930),   by the Swiss State Secretariat for Education, Research and Innovation (SERI) under contract number MB22.00034, and by the AEI project PID2021-125021NA-I00 (Spain).}

\begin{document}

\begin{abstract}
We prove new boundary regularity results for minimizers to the one-phase Alt-Caffarelli functional (also known as Bernoulli free boundary problem) in the case of continuous and Hölder-continuous boundary data. As an application, we use them to extend recent generic uniqueness and regularity results to families of continuous functions.
\end{abstract}

\maketitle

\section{Introduction}
In this work, we study minimizers of the \textit{Alt--Caffarelli functional}
\begin{equation}
\label{fl}
    \fl (u,D) \coloneq  \int_{D} |\nabla u|^2 dx + \Lambda |\{u>0\} \cap D|, 
\end{equation}
where $D$ is an open domain in $\R^d$, $\Lambda$ a positive real constant and $u\in H^1(D)$.

The problem, also known as one-phase or Bernoulli problem, originates in the fundamental works \cite{caffarellifirst, caffarellioriginalC1alpha} and has various applications in models of flame propagation \cite{applicationflame} and jet flows \cite{applicationsjetflow}.
Recent development include \cite{criticaldimensionccurrent, recent3, velichkovpaper, Engelstein2022GraphicalST, fernandezreal2023generic}. 
We also refer to \cite{caffarellisalsa} and \cite{velichkov} for a detailed mathematical exposition.

Given an open bounded domain $D\subset \R^d$ and a boundary datum $g\in H^1(D)$ with $g\geq 0$ in $D$, the Alt--Caffarelli problem is the  minimization:
\begin{equation}
\label{minimizationproblem}
    \min \{ \fl (u,D): u\in H^1(D) \text{ such that } u-g \in H_0^1(D)\}.
\end{equation}
Any minimizer $u$ is nonnegative, and splits the domain into two parts:
\[
\Omega_u:=\{x \in D :u(x)>0\}\quad\text{ and } \quad \Omega_0:=\{x \in D :u(x)=0\}.
\]
The interface between the two sets, $\partial \Omega_u$, is a priori unknown and is called the \textit{free boundary}.

Any minimizer $u$ to \eqref{minimizationproblem} is locally Lipschitz continuous inside $D$ (see e.g.\ \cite[Chapter 3]{velichkov}). The Euler--Lagrange equation, satisfied by (classical) stationary points of \eqref{fl}, is given by
\[
\left\{
\begin{array}{rcll}
u & \ge & 0 & \quad\text{in}\quad D,\\
\Delta u & = & 0 & \quad\text{in}\quad  \Omega_u,\\
|\nabla u | & = & \sqrt{\Lambda}& \quad\text{on}\quad \partial\Omega_u\cap D.
\end{array}
\right.
\]
For general stationary solutions, the previous equations need to be understood in the viscosity sense. Throughout the paper, for the sake of simplicity, we will fix $\Lambda = 1$.

Reminiscent of the classical Laplace equation with Dirichlet boundary condition\footnote{That is, on a sufficiently regular domain $D$, given a boundary datum $g$, the solution $u$ ``inherits'' (in some sense) the regularity of $g$.}, the main goal of this work is to establish basic regularity estimates up to the boundary for solutions to \eqref{minimizationproblem}. To our knowledge, up until now, the community has been proving such estimates on a need-to-use basis (see \cite{velichkovpaper,fernandezreal2023generic}). We hope that this short note can be useful to researchers in contexts where such estimates can be applied. In this direction, we present some examples of applications of our results, namely a comparison principle and generic-type results for minimizers.

\subsection{Main results}
Our main result says that minimizers of the one-phase problem with continuous boundary datum are continuous up to the boundary. The following result applies, for example, to the case of $C^1$ domains. 

\begin{theorem}
\label{theoremcontinuityuptotheboundary}
     Let $d\geq 2$ and $D \subset \R^d$ be an open domain such that 
    \begin{itemize}
        \item either $D$ is convex,
        \item or $D$ is a locally $c$-Lipschitz domain, for some $c$ small enough  depending only on  $d$. 
    \end{itemize}
    Let $g\in C(\overline D) \cap H^1(D)$ with  modulus of continuity $\omega$ and $\|g\|_{H^1(D)}\le M$ for some $M > 0$, and let $u$ be a minimizer to \eqref{minimizationproblem}. 
    
    Then, $ u \in C(\overline D)$, with a modulus of continuity depending only on $\omega$, $D$, $\Lambda$, and $M$. That is, for any $\eps>0$, there exists $\delta=\delta(\omega, D, \Lambda, M)$ such that
    \begin{equation*}
        |x-y| < \delta \implies |u(x)-u(y)| < \eps \qquad \forall x,y \in \overline D.
    \end{equation*}
\end{theorem}

A priori, as for the case of  harmonic functions, the modulus of continuity of $u$ does not need to be the same (nor comparable) to the modulus of continuity of $g$. 

This is in contrast to the case of more regular boundary data, where for H\"older coefficients we actually obtain H\"older regularity up to the boundary (again, as for harmonic functions):
\begin{proposition}
\label{prop:holder}
Let $d\geq 2$ and $D \subset \R^d$ be an open bounded $C^{1,\alpha}$ domain.
    Let $g\in C^\gamma(\bar D) \cap H^1(D)$ where $\frac{1}{2}<\gamma<1$, and let $u$ be a minimizer to \eqref{minimizationproblem}. Then, $ u \in C^\gamma(\overline D)$ and
    \begin{equation*}
        \|u\|_{C^{\gamma}( \overline D)} \leq  C \left(1+ \|g\|_{C^{\gamma}(\partial D )}+ \|u\|_{L^\infty( D)} \right),
    \end{equation*}
    where $C$ depends only on $d$, $\gamma$, $\Lambda$, $\alpha$ and $D$ (in particular, through its diameter and $C^{1,\alpha}$ norm; see Definition~\ref{c1alphadef}).
\end{proposition}

The previous result is a generalization of the case for $\gamma = 1$, originally treated in \cite[Appendix B]{velichkovpaper}. 

\subsection{Applications to generic regularity}
In the second part of the paper, we apply the continuity up to the boundary to show how to extend the results on generic uniqueness of minimizers from \cite{fernandezreal2023generic} to the case of merely continuous data. 

Namely, we show:

\begin{proposition}
\label{theoremgenericuniqueness}
    Let $d \ge 2$, and let $D\subset \R^d$ be a domain as in Theorem~\ref{theoremcontinuityuptotheboundary}. 
    Let  $g_t\in C(\overline{D})\cap H^1(D)$ for $t\in (0, 1)$ with $\sup_{t\in (0, 1)}\|g_t\|_{H^1(D)}< \infty$ be such that,    for all  $0 < s< t< 1$,
    \begin{enumerate}[label=(\roman*)]
    \item $g_t\ge g_s\ge 0$ in $D$, and
    \item  any connected component of $\{g_s > 0\}\cap \partial D$ contains  $x_0$ with $g_t(x_0) > g_s(x_0)$.
    \end{enumerate}
    
 Then, there exists a countable subset $J \subset (0,1)$ such that
    for every $t \in (0,1) \backslash J$, there is a unique minimizer of $F_\Lambda(\cdot, D)$ with boundary datum given by $g_t$.
\end{proposition}

Lastly, we also show a generic regularity result for the free boundary. By  \cite{Weiss1999PartialRF}, it is already known that up to a certain critical dimension $d^*$ ($4\leq d^* \leq 6$, see \cite{criticaldimensionccurrent}) the free boundary of $u$ is always smooth, i.e.\ its set of singular points, denoted ${\rm Sing}(u)$, is empty (and in general dimension, it has Hausdorff dimension $d-d^*-1$). Thanks to \cite{fernandezreal2023generic}, generically this dimension can be increased by one if one takes minimizers with Lipschitz boundary data. We generalize the result to a wider class of boundary data:

\begin{theorem}
\label{theoremgenericregularity notequicontinuous}
    Let $d \ge 2$, and let $D\subset \R^d$ be a domain as in Theorem~\ref{theoremcontinuityuptotheboundary}. 
    Let  $g_t\in C(\overline{D})\cap H^1(D)$ with $g_t \ge 0$ for $t\in (0, 1)$,  $\sup_{t\in (0, 1)}\|g_t\|_{H^1(D)} <\infty$, and  
    \[
    g_t - g_s \ge t-s\quad \text{in}\quad  D \quad\text{for all}\quad  0< s< t< 1.
    \]
    
    Let $u_t$ denote any minimizer of $F_\Lambda(\cdot, D)$ with boundary datum $g_t$. Then:
    \begin{itemize}
        \item If $d=d^*+1$,  there exists a countable subset $J \subset (0,1)$ such that 
        \begin{equation*}
            {\rm Sing}(u_t) = \varnothing \qquad \qquad \text{for every } t \in (0,1) \backslash J.
        \end{equation*}
        \item If $d\geq d^*+2$,
        \begin{equation*}
            \textnormal{dim}_{\Hh} \, {\rm Sing}(u_t) \leq d- d^*-2 \quad\quad \text{for almost every } t \in (0,1),
        \end{equation*}
    \end{itemize}
    where ${\rm dim}_{\mathcal{H}}$ denotes the Hausdorff dimension of a set. 
\end{theorem}

\begin{remark}
    As an example, the family $\{g+\lambda\}_{\lambda \in (0,1)}$ with $g:\partial D \to \R$ nonnegative and continuous, satisfies the assumptions of Proposition~\ref{theoremgenericuniqueness} and Theorem\ref{theoremgenericregularity notequicontinuous}. 
\end{remark}

\begin{remark}
 Contrary to \cite{fernandezreal2023generic}, where the family ${g_t}$ is required to be equi-Lipschitz continuous, any assumption on equicontinuity becomes redundant and only uniform boundedness of the family ${g_t}$ and monotonicity are required.
\end{remark}

We finally also refer to Lemma~\ref{orderedwrtdatum} for a comparison principle between minimizers with continuous boundary data.

\subsection{Structure of the paper}
We start by proving, in Subsection~\ref{sectioncontinuous} and by means of a barrier and compactness argument, the quantitative continuity up to the boundary, Theorem~\ref{theoremcontinuityuptotheboundary}.
In Subsection~\ref{sectionhölder}, we then show Proposition~\ref{prop:holder}:  for Hölder continuous boundary datum, the minimizer is also Hölder continuous (with the same exponent) up to the boundary. For that, we use a modified version of the arguments in \cite[Lemma~B.1]{velichkovpaper}. 

Finally, in Section~\ref{sectiongeneric}, we apply our results by first proving a general comparison lemma for continuous minimizers, and then  to show generic uniqueness (Proposition~\ref{theoremgenericuniqueness} in Subsection~\ref{sectionuniqueness}) and generic regularity (Theorem~\ref{theoremgenericregularity notequicontinuous} in Subsection~\ref{sectionregularity}). There we show how to modify the arguments from \cite{fernandezreal2023generic}, and how to work around the equicontinuity of the boundary data.

\section{Boundary regularity}
\label{section1}
This section introduces the two new boundary regularity results.
Note that the regularity of the ambient domain is crucial as well, however here we are not concerned with necessary conditions (\textit{Wiener-type criterioa}) and assume sufficient regularity of $\partial D$ as needed. 
We recall that the hypograph of a function $f:\R^d \to \R$ is  given as 
\begin{equation*}
    \hyp(f) := \{(x,y) \in \R^{d+1}: f(x) \ge y \}.
\end{equation*}
We also state some standard definitions here for the reader's convenience.

\begin{definition}
    A domain $D \subset \R^d$ is $c$-Lipschitz for some $c > 0$, if for each $x_0\in \partial D$, up to a rotation, $\partial D$ is the graph of a $c$-Lipschitz function $\varphi$ in $B_1(x_0)$.
\end{definition}
\begin{definition}
\label{c1alphadef}
    A domain $D\subset \R^d$ is a $C^{1,\alpha}$ domain for some $\alpha\in (0, 1]$, if for each $x_0\in \partial D$, up to a rotation, $\partial D$ is the graph of a $C^{1,\alpha}$ function $\varphi$ in $B_1(x_0)$. The maximum $C^{1, \alpha}$ norm of such function among all $x_0\in \partial D$ is what we call the $C^{1,\alpha}$ norm of the domain $D$. 
\end{definition}
\begin{remark}
    Up to a rescaling, any bounded domain that is locally Lipschitz/$C^{1,\alpha}$ is a Lipschitz/$C^{1,\alpha}$ domain respectively.
\end{remark}

\subsection{Continuous boundary datum}
\label{sectioncontinuous}

The first result we prove concerns the regularity of minimizers with merely continuous datum. We recall the well-known solution on an annulus. Remember that we are fixing $\Lambda = 1$.
\begin{proposition} \textnormal{\cite[Proposition 2.15]{velichkov}}
\label{solutioninannulus}
    Let $d\ge 2$. There exists $R=R(d)\in (1, 2)$ such that  
    \begin{equation*}
        v_1(x) := \begin{cases}       
        1-\frac{\log |x|}{\log R} &\text{ if $d=2$},\\[7pt] 
        \frac{|x|^{2-d} - R^{2-d}}{1 - R^{2-d}} &\text{ if $d\geq3$},\\
        \end{cases}
    \end{equation*}
    is the unique solution of $\eqref{minimizationproblem}$ on $A= B_R\backslash B_1$ with $g|_{\partial B_1}=1$ and $g|_{\partial B_R}=0$.
\end{proposition}


Using the previous explicit solution as a barrier, we are able to prove quantitatively that the minimizer $u$ is continuous up to the boundary (Theorem~\ref{theoremcontinuityuptotheboundary}). The modulus of continuity of the solution $u$ is not necessarily the same as for the boundary datum, but depends on it (as well as the domain and the modulus of the boundary datum itself). We start with a lemma, stating that minimizers are positive close to positive boundary data:

\begin{lemma}
\label{lem:basic_used}
    Let $d\ge 2$, and let $D$ be an open domain such that 
    \begin{itemize}
        \item either $D$ is convex,
        \item or $D$ is  $c$-Lipschitz, for some $c$ small enough depending only on $d$. 
    \end{itemize}
    Let us assume, moreover, that $0\in \partial D$, and that $u \ge 2 > 0$ on $\partial D\cap B_2$. Then,
    \[
    u > 0 \quad\text{in}\quad D\cap B_\rho,
    \]
    for some $\rho > 0$ depending only on   $d$. 
\end{lemma}
\begin{proof}
We proceed with a barrier argument (see Figure~\ref{fig:setup} for a sketch of the setting in the two types of domain). 
\begin{figure}
    \centering
    \begin{minipage}[t]{0.49\textwidth}
        \centering
        \includegraphics[width=0.8\linewidth, trim=0 140 0 0, clip]{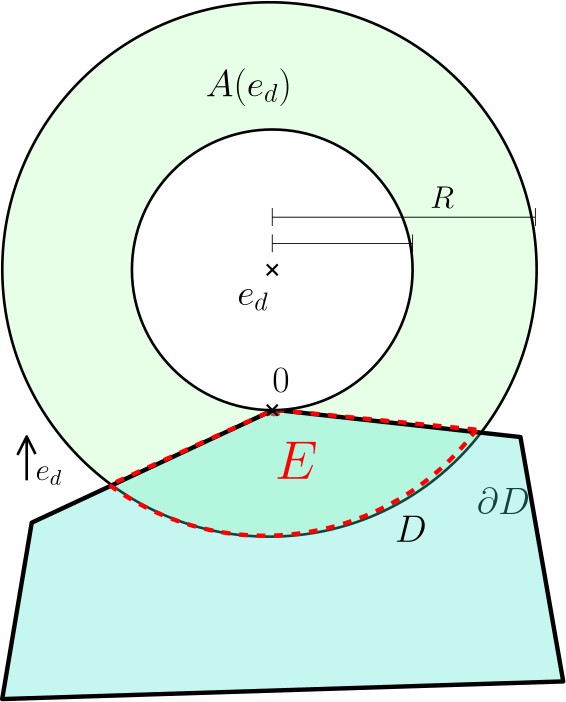}
        \caption*{a) $D$ is convex}
        \label{fig:image1}
    \end{minipage}
    \hfill
    \begin{minipage}[t]{0.49\textwidth}
        \centering
        \includegraphics[width=0.8\linewidth, trim=230 160 160 0, clip]{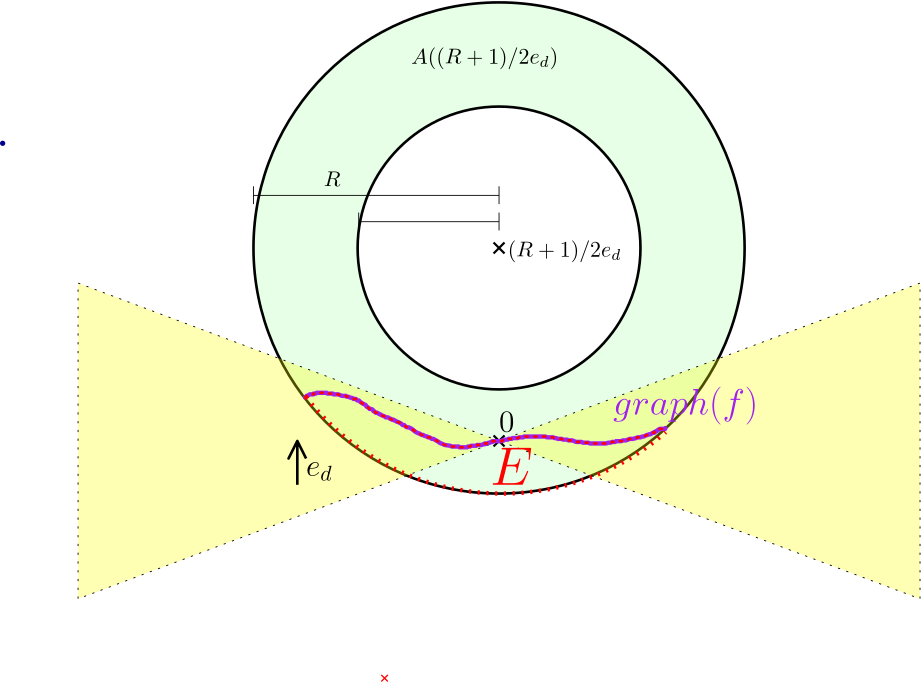}
        \caption*{b) $D$ is $c$-Lipschitz}
        \label{fig:image2}
    \end{minipage}
    \caption{The set-up for the proof of Theorem~\ref{theoremcontinuityuptotheboundary}}
    \label{fig:setup}
\end{figure}
Take the annulus $A$ from Proposition~\ref{solutioninannulus}. We now consider the two cases:
\begin{itemize}
    \item \textbf{$D$ is convex:} Up to a rigid motion, we assume that $D \subset H$, where $H = \{(x',x_d): x_d\leq 0\}$.
        Set 
    $$ E:= A(e_d) \cap D, \qquad \partial E = E_1 \cup E_2
    $$ 
    where
    $$
    E_1 := \partial D \cap A(e_d) \subset \partial D \cap B_1, \qquad E_2 :=\partial B_{R} (e_d) \cap D.
    $$

    Let $v_1$ be the (unique) solution from Proposition~\ref{solutioninannulus}   in the annulus $A$.      Then for $ v(x):= v_1(x+ e_d)$, on $E_1$,
    $$
     v(x)\leq 1 <2 \le u(x)\quad\text{on}\quad E_1,
    $$
    and $ v(x) =0$ on $E_2$.
    In particular, we get that 
    $$
    v\le u\quad\text{on}\quad \partial E.
    $$
    
    \item \textbf{$\partial D$ is $c$-Lipschitz:}  Without loss of generality (up to rotation and rescaling), we assume that $  D$ is the subgraph/hypograph   of a $c$-Lipschitz function $f:B^{d-1}_2 \to \R$ in the $e_d$ direction and $f(0)=0$ (denote here by $B_2^{d-1}$ the $(d-1)$-dimensional ball of radius $2$).
    Since $f$ is $c$-Lipschitz, we have that  for any $0<h<1$ the graph of $f$ lies within $B_h^{d-1}\times [-ch,ch]$.
   
        Then (recall $R\in (1, 2)$ is given by Proposition~\ref{solutioninannulus}, so $\frac{R+1}{2}\in (1, R)$), we have 
    $$
    \graph(f) \cap A\left(\frac{R+1}{2} e_d \right)  \subset B_{2}^{d-1} \times [-2c, 2c].
    $$
    Set now $E$ to be the connected component of 
    $$\hyp(f) \cap A\left(\frac{R+1}{2}e_d \right)$$
    containing the origin.
   As long as $c\leq \frac{R-1}{2}$, $\partial E = E_1 \cup E_2$ with
    $$E_1 \subset \graph(f)\quad  \text{ and }\quad E_2 \subset \partial B_{R}\left(\frac{R+1}{2} e_d \right). $$
         
         Let $v_1$ be the solution on the annulus $A$, for 
    $$v(x):= v_1 \left( x+\frac{(R+1)}{2}e_d \right)$$
    we have
    $$
    v(x)\leq 1 <2 \le u(x) \qquad \text{ on } E_1
    $$ 
    and $v(x)=0$ on $E_2$. In particular, we get again that 
    $$
    v\le u \quad\text{on}\quad \partial E.
    $$
\end{itemize}
In both cases, we have constructed a set $E$ where the boundary consists of two parts $E_1$ and $E_2$, and $u|_{E_1} > v|_{E_1}$ and $u|_{E_2}\geq v|_{E_2}=0$. Also $v|_E$ and $u|_E$ are minimizers on $E$ for their own boundary datum, i.e.\
$$
\fl(v,E) \leq \fl(\min(v,u), E) \qquad \text{and} \qquad \fl(u,E) \leq \fl(\max(v,u), E).
$$
From the cut-and-paste lemma for minimizers to the one-phase problem (see \cite[Lemma~2.5]{velichkov}),
\begin{equation*}
    \fl(\min(v,u),E) + \fl(\max(v,u),E) = \fl (v,E) + \fl(u,E),
\end{equation*}
thus $\fl(v,E)= \fl(\min(v,u),E)$. Since $v$ as a minimizer is unique, we have $v=\min(v,u)$, i.e.\ $u\geq v$ with $v$ vanishing only on a subset of $\partial E$.
Hence, there exists some small $\rho>0$ such that on $B_\rho \cap D$ the function $u$ is strictly positive as was to be shown.
 \end{proof}

As a consequence we obtain the proof of the regularity up to the boundary:

\begin{proof}[Proof of Theorem~\ref{theoremcontinuityuptotheboundary}]
    Assume by contradiction that it is not true. Then there exists $\bar \eps>0$ and a sequence $\{g_k\}_{k\in \mathbb{N}}$ with $g_k \in C(\overline D) \cap H^1(D)$ having a uniform modulus of continuity $\omega$ and $\|g_k\|_{H^1(D)} \leq M<\infty$ such that for some minimizers $u_k$ to \eqref{minimizationproblem} with $u_k-g_k=H_0^1(D)$, there exist $x_k,y_k \in \overline D$ such that 
    \begin{equation}
    \label{thingtocontradictagainst}
     |x_k-y_k| \leq \frac{1}{k}\to 0 \qquad \text{but} \qquad   |u_k(x_k)-u_k(y_k)| \geq \bar \eps >0,\qquad\text{for}\quad k\in \mathbb{N}. 
    \end{equation}
    
By the uniformity of the modulus of continuity $\omega$ and the boundedness of the $H^1$-norm of $g_k$, again up to a subsequence   $g_k \to g_\infty$ for some $g_\infty$ with the same modulus of continuity $\omega$ and $\|g_\infty\|_{H^1(D)}\leq M$. Note also that $u_k \to u_\infty$ in $H^1_{loc}(D)$ with $u_\infty$ being a minimizer of \eqref{minimizationproblem} with  boundary datum $g_\infty$ by \cite[Lemma~6.3]{velichkov}. Both, $u_k$ and $u_\infty$ are locally Lipschitz continuous in $D$ independently of $k$.

Moreover, by compactness of $D$, $x_k$ and $y_k$ converge, up to a subsequence, to $x_0 \in \overline D$. We now separate between three cases: \\

\noindent \textbf{Case $x_0 \in D$:} For sufficiently large $k$, $x_k$, $y_k$, and $x_0$ are inside some $D'\subset \subset D$. By the interior uniform (Lipschitz) continuity of $u_k$ and $u_\infty$ 
we get a contradiction with~\eqref{thingtocontradictagainst}.\\

\noindent \textbf{Case $x_0 \in \partial D \cap \{g_\infty=0\}$:} Since $g_k\to g_\infty$, $g_k(x_0) \to 0$ as well. Consider now the solution to the Dirichlet problem,
\begin{equation*}
\left\{
    \begin{array}{rlll}
        \Delta \bar u_k &=&0 \qquad &\text{in }D\\
        \bar u_k &=& g_k \qquad &\text{on }\partial D.
    \end{array}
    \right.
\end{equation*}
By the classical theory (\cite[Lemma~2.13]{GilbargTrudinger}), the $\bar u_k$'s have the same modulus of continuity $\bar \omega$, depending only on $\omega$, $d$, $D$ and $M$. Thus
\begin{equation*}
    | \bar u_k(x_k)| \leq |\bar u_k(x_k)-\bar u_k(x_0)| + |g_k(x_0)| \leq \bar \omega(|x_k-x_0|) + |g_k(x_0)|,
\end{equation*}
which vanishes as $k \to \infty$. Hence $\bar u_k(x_k), \bar u_k(y_k)\to 0$. On the other hand, from the subharmonicity of $u_k$ , by the comparison principle for weak (sub)solutions we have
\begin{equation*}
    |u_k(x_k) - u_k(y_k)| \leq u_k(x_k) + u_k(y_k) \leq \bar u_k(x_k) + \bar u_k(y_k) \to 0,
\end{equation*}
again a contradiction with \eqref{thingtocontradictagainst}.\\

\noindent \textbf{Case $x_0 \in \partial D \cap \{g_\infty>0\}$:} We proceed by using the barrier argument from Lemma~\ref{lem:basic_used}. Without loss of generality, up to a translation, we assume $x_0=0$ and observe that for some $\rho >0$ and for any $k$ sufficiently large, $u_k>0$ in  $B_\rho \cap D$, and therefore, the $u_k$'s are harmonic there.

Indeed, for $k$ sufficiently large, and up to a rescaling by $r$, $\frac{u_k(r x)}{r}$  (independent of $k$), we can assume that   $\frac{g_k}{r} \geq \frac{g_\infty(0)}{2 r}$ in $B_2 \cap \partial D$, so that, up to taking $r$ smaller if necessary (such that $\frac{g_\infty(0)}{2 r}\geq 2$), we are in the setting of Lemma~\ref{lem:basic_used}. Thus, there is a small $\rho > 0$ (independent of $k$) such that $u_k > 0$ in $D\cap B_\rho$,   and thereby the $u_k$'s are harmonic there. We apply \cite[Lemma~2.13]{GilbargTrudinger} to get $u_k$ continuous in $\overline{B_{\rho/2} \cap D}$ with a common modulus of continuity $\bar \omega$, depending only on $\omega$, $d$, $\rho$, and $D$.

Thus 
$$|u_k(x_k)-u_k(x_0)| \leq \bar \omega(|x_k-x_0|), \qquad |u_k(y_k)-u_k(x_0)|\leq  \bar \omega(|y_k-x_0|)$$
and therefore, as $k \to \infty$, by the triangle inequality
\begin{equation*}
    |u_k(x_k)-u_k(y_k)| \leq \bar \omega(|x_k-x_0|) + \bar \omega(|y_k-x_0|) \to 0,
\end{equation*}
again a contradiction to \eqref{thingtocontradictagainst}.
This finishes the proof.
\end{proof}

\subsection{Hölder continuous boundary datum}
\label{sectionhölder}
For the case with Hölder continuous boundary datum, we show that the Hölder regularity is preserved. We do so by following \cite{velichkovpaper}. First we state a well-known technical tool, the Morrey Lemma \cite[Lemma~3.12]{velichkov}.
\begin{lemma}
\label{morrey}
    Let $\Omega \subset \R^d$, $u\in H^1(\Omega)$ and for $C>0$, $\gamma \in (-1,1)$
\begin{equation*}
    \fint_{B_r(x_0)} |\nabla u|^2 dx \leq C^2 r^{2(\gamma-1)} \qquad \forall x_0 \in B_{R/8}  \quad \forall r \leq \frac{R}{2}.
\end{equation*}
Then $u \in C^{0,\gamma}(B_{R/8})$ with $$\|u\|_{C^{0,\gamma}(B_{R/8})} \leq C \left(2^d + \frac{2}{\gamma} \right). $$
\end{lemma}

We now can prove the local version of Proposition~\ref{prop:holder}: 

\begin{proposition}
\label{ballHölder}
Let $D$ be an open bounded $C^{1,\alpha}$ domain in $B_1$, given by the subgraph of a function with $C^{1,\alpha}$ norm bounded by 1.  
    Let $u$ be a minimizer of \eqref{minimizationproblem} on $D$ with boundary datum $g \in C^{\gamma}(\bar D)\cap H^1(D)$ with $1>\gamma>\frac{1}{2}$, and $0\in \partial\{u > 0\}$. 
    Then,
    \begin{equation*}
        \|u\|_{C^{\gamma}(B_{1/2} \cap \overline D)} \leq  C \left( \|u\|_{C^{\gamma}( B_1 \cap \partial D )}+ \|u\|_{L^\infty(B_1  \cap D)} +1\right),
    \end{equation*}
    for some constant $C$ depending only on $d$, $\gamma$, and $\alpha$.
\end{proposition}

\begin{proof}
Let $\partial D$ be the graph of a $C^{1,\alpha}$ function $\phi$ in $B_1$, i.e.\ $\partial D \cap B_1 = \{(x',\phi(x')): x' \in \R^{d-1}\}$. 
It suffices to show $\gamma$-Hölder continuity in a small ball $B_{1/8}$.
Denote the positive and negative half spaces by $H^+$ and $H^-$. 

First, we extend $\nabla u:B_1 \cap D \to \R^d$ to $\nabla u:B_1 \to \R^d$. In order to do so, let $\Phi:B_1 \to \R^d$ be the $C^{1,\alpha}$ function
\begin{equation*}
    \Phi(x',x_d) := (x', x_d + g(x')).
\end{equation*}
 Up to translation and rotation, we assume $\Phi(0)=0$ and $D\Phi(0)=I_{d\times d}$.
Let 
$$\pi:(y',y_d) \mapsto (y',-y_d),$$
we define the extension
\begin{equation*}
    \nabla u (x)  = \begin{cases}
        \nabla u (x) \qquad &\text{if } x \in D \cap B_1,\\
        \nabla u ( \Phi \circ \pi \circ \Phi^{-1}( x)) \qquad &\text{if } x \in D^c \cap B_1.
     \end{cases}
\end{equation*}

    The idea is to arrive at an estimate for $u$ of the form
    \begin{equation}
    \label{morreyassumption}
      \forall \bar x \in B_{1/8}, \forall r\leq \frac{1}{2}: \qquad  \int_{B_r(\Bar x)} |\nabla u|^2 \leq Cr^{d+2(\gamma-1)},
    \end{equation}
    and then use the Morrey Lemma \cite[Lemma~3.12]{velichkov}, which gives directly $\gamma$-Hölder regularity in $B_{1/8}$.
   As in \cite{velichkovpaper}, it suffices to show the estimate on the boundary, i.e.\ for a fixed $r>0$ small enough and $x_0 \in \partial D \cap B_{1/8}$,
\begin{equation}
\label{eq:Cbound}
\int_{B_r(x_0)} |\nabla u|^2 \leq Cr^{d+2(\gamma-1)}.
\end{equation}
By a translation, we assume that $x_0=0$.
 Performing the change of variable, 
 $$x=\Phi(y),\qquad  dx=|D\Phi(y)|dy,$$
 we have
\begin{equation*}
\begin{split}
    &\int_{B_r} |\nabla u(x)|^2 dx =\int_{\Phi(B_r)} |\nabla u(x)|^2 dx = \int_{B_r} |\nabla u(\Phi(y))|^2 |D\Phi(y)| dy\\
    &\qquad =  \int_{B_r \cap H^+} |\nabla u(\Phi(y))|^2 |D\Phi(y)| dy + \int_{B_r \cap H^-} |\nabla u(\Phi \circ \pi \circ \Phi^{-1}( \Phi(y)))|^2 |D\Phi(y)| dy \\
    &\qquad =  2\int_{B_r \cap H^+} |\nabla u(\Phi(y))|^2 |D\Phi(y)| dy = 2 \int_{\Phi(H^+ \cap B_r)} |\nabla u(x)|^2 dx.
\end{split}
\end{equation*}

\noindent \textbf{Step 1: }Let $h_g:H^+ \cap B_1 \to \R$ such that
\begin{equation*}
\left\{
    \begin{array}{rlll}
        \Delta h_g &= &0 \qquad &\text{in } H^+ \cap B_1, \\
        h_g &= &g \circ \Phi \qquad &\text{ on } \partial(H^+ \cap B_1),
    \end{array}
    \right.
\end{equation*}
which is in $C^{\gamma}(\bar H^+ \cap B_{1/2})$ by \cite[Proposition~2.1]{Milakis2006RegularityFF} ($u$ is continuous in $\overline D$ by Theorem~\ref{theoremcontinuityuptotheboundary}) with 
\begin{equation*}
\begin{split}
\|h_g\|_{C^{\gamma}(H^+\cap B_{1/2})} &\leq C_{d,\gamma} \left(\|g \circ \Phi\|_{C^{\gamma}(H^+ \cap B_1)} + \|h_g\|_{L^\infty( \partial H^+ \cap B_1)} \right) \\
&\leq C \left(\|g \|_{C^{\gamma}(\partial D)}  + \|g\|_{L^\infty(H^+ \cap B_1)} \right),
\end{split}
\end{equation*}
for $C$ depending on $d,\gamma$ and the $C^{1,\alpha}$ norm of $D$. We now claim that inside $H^+ \cap B_{1/2}$
$$
|\nabla h_g(x)|\leq C_{d} \|h_g\|_{C^{\gamma}(H^+\cap B_{1/2})} \dist(x,\partial H^+)^{\gamma-1}.
$$ 
For a fixed $x$, Take $x_1\in \partial H^+$ such that $ |x-x_1|=\dist(x,\partial H^+) =: \delta(x)$ and let  
$$
v(x)\coloneq h_g(x)-h_g(x_1).
$$
By $\gamma$-Hölder regularity of $v$, from the definition, for $x \in H^+ \cap B_{1/2}$
\begin{equation*}
    v(x) = v(x_1) + R(x)=R(x) \qquad \text{with } |R(x)|\leq 2 \|h_g\|_{C^{\gamma}(H^+\cap B_{1/2})}|x-x_1|^{\gamma}.
\end{equation*}
By standard harmonic estimates,
\begin{equation*}
\begin{split}
|\nabla h_g(x)|= |\nabla v(x)| &\leq \sup_{B_{\delta(x)/4}(x)}|\nabla v(y)| \\
&\leq \frac{d}{\delta(x)} \sup_{B_{\delta(x)/2}(x)} |v(y)| \\
&\leq \frac{d}{\delta(x)} \sup_{B_{\delta(x)/2}(x)} |R(y)| \\
&\leq \frac{2 d\|h_g\|_{C^{\gamma}(H^+\cap B_{1/2})}}{\delta(x)} \sup_{B_{\delta(x)/2}(x)}|y-x_1|^{\gamma} \\ &\leq 4 d \|h_g\|_{C^{\gamma}(H^+\cap B_{1/2})} \delta(x)^{\gamma-1},
\end{split}
\end{equation*}
   proving the claim. In the co-area formula \cite[Theorem~3.2.22]{federer}, take $f(x)= \delta (x)$ and $g(x) = \delta (x)^{2(\alpha-1)}$. Since $\partial H^+ =\{x_d=0\}$ we have $|\nabla f(x)|=1$. Hence, as $$f^{-1}(t)=\{ x \in H^+\cap B_r: \delta (x)=t\},$$ we estimate for $r< \frac{1}{2}$,
    \begin{equation*}
    \begin{split}
    \int_{H^+\cap B_r}|\nabla h_g|^2dx &\leq C_{d} \|h_g\|^2_{C^{\gamma}(H^+\cap B_{1/2})}
        \int_{H^+\cap B_r} \delta (x)^{2(\gamma-1)}dx\\
        &=C_{d} \|h_g\|^2_{C^{\gamma}(H^+\cap B_{1/2})} \int_{H^+\cap B_r} \delta (x)^{2(\gamma-1)} |\nabla f|dx\\
        &=C_{d} \|h_g\|^2_{C^{\gamma}(H^+\cap B_{1/2})} \int_{\R} \int_{\{ x\in H^+\cap B_r: \delta (x)=t\}} \delta (x)^{2(\gamma-1)} d\Hh^{d-1}(x) dt\\
        &=C_{d} \|h_g\|^2_{C^{\gamma}(H^+\cap B_{1/2})} \int_{\R} t^{2(\gamma-1)} \Hh^{d-1}(\{ x\in H^+\cap B_r: \delta (x)=t\}) dt \\
        &=C_{d} \|h_g\|^2_{C^{\gamma}(H^+\cap B_{1/2})} \int_0^r t^{2(\gamma-1)} \Hh^{d-1}(\{ x\in H^+\cap B_r: \delta (x)=t\}) dt \\
        &\leq C_{d} \|h_g\|^2_{C^{\gamma}(H^+\cap B_{1/2})} r^{d-1} \int_0^r t^{2(\gamma-1)} dt \\ 
        &= C_{d} \|h_g\|^2_{C^{\gamma}(H^+\cap B_{1/2})}r^{d+2(\gamma-1)}.
    \end{split}
    \end{equation*}
The rest of the proof follows exactly as in \cite{velichkovpaper}, with the constant only depending on $d$, $\gamma$ and the $C^{1,\alpha}$ norm of $D$, but not on the Hölder norm of $g$.
(We remark that $g$ here is named $\varphi$ in \cite{velichkov}.)
We thereby conclude that \eqref{morreyassumption} is satisfied since
\begin{equation*}
    \int_{B_r(x_0)} |\nabla u|^2 \leq  C  \|h_g\|^2_{C^{\gamma}(H^+ \cap B_{1/2})} r^{d+2(\gamma-1)},
\end{equation*}
and thus by Lemma~\ref{morrey}, the analogous interior regularity estimate, and the bound for $\|h_g\|_{C^\gamma(H^+\cap B_{1/2})}$, we have that $u$ is locally $C^{\gamma}$ Hölder continuous with 
\begin{equation*}
    \|u\|_{C^{\gamma}(B_{1/8}\cap \overline D)} \leq  C \left( \|u\|_{C^{\gamma}( B_1(x_0)\cap \partial D )}+ \|u\|_{L^\infty(B_1(x_0) \cap D))}+1 \right),
\end{equation*}
 as we wanted to show.
\end{proof}

\noindent As a consequence, we obtain directly Proposition~\ref{prop:holder}.

\begin{proof}[Proof of Proposition~\ref{prop:holder}]
    Since $D$ is bounded and a $C^{1,\alpha}$ domain, up to a rescaling we can pick $x_1,...,x_N \in \partial D$ such that $\partial D \subset \bigcup_{i=1}^N B_{1/2} (x_i)=:S$ and inside each $B_{1 }(x_i)$ we are in the setting of Proposition~\ref{ballHölder}. Thus by Proposition~\ref{ballHölder}, 
    \begin{equation*}
    \begin{split}
        \|u\|_{C^\gamma(S)} &\leq \sum_{i=1}^N  \|u\|_{C^{\gamma}(B_{1/2}(x_i)\cap \overline D)} \\&\leq \sum_{i=1}^N C \left( \|u\|_{C^{\gamma}( B_1(x_i)\cap \partial D )}+ \|u\|_{L^\infty(B_1(x_i) \cap D))} +1\right) \\
        &\leq C \left( \|g\|_{C^{\gamma}(\partial D )}+ \|u\|_{L^\infty( D)} +1\right).
    \end{split}
    \end{equation*}
    Since $u$ is Lipschitz continuous on $D\backslash S \subset \subset D$  with 
    \begin{equation*}
        \|\nabla u\|_{L^\infty (D\backslash S)} \leq C \left( 1+ \frac{\|u\|_{L^\infty(D)}}{\dist^{d+1}(D\backslash S, \partial D)} \right),
    \end{equation*}
    the result follows.
\end{proof}

\begin{remark}
For $\gamma< \frac{1}{2}$, it is not a priori given that the boundary datum $g$ is the trace of a function in $H^{1}(D)$. For $\gamma\leq \frac{1}{2}$ and $g\in C^{\gamma}(\bar D) \cap H^{1}(D)$, the previous proof does not work, since we crucially use the minimality if $u$.
\end{remark}

\begin{remark}
    If $\gamma=1$ (i.e.\ the datum is Lipschitz), then using the same argument as above we recover the result from \cite{velichkovpaper}, as expected,
        \begin{equation*}
       \forall x_0\in \overline \Omega \cap B_{1/2}, r\leq \frac{1}{2}: \qquad \fint_{B_r(x_0)} |\nabla u|^2 \leq Cr^{d+2(\gamma-1)}, \qquad \forall \gamma<1.
    \end{equation*}
    This implies local $\gamma$-Hölder regularity for any $\gamma<1$, but not Lipschitz regularity, in the exact same fashion as for the Laplace equation with Dirichlet boundary condition. It remains open whether this result could be improved to show e.g.\ $\log$-Lipschitz continuity of the solution,
    \begin{equation*}
        |u(x)-u(y)| \leq C |\log|x-y||\cdot
        |x-y | \qquad \forall x,y\in \overline D.
    \end{equation*}
\end{remark}

\section{Applications of boundary regularity}
\label{sectiongeneric}
\subsection{Generic uniqueness of minimizers}
\label{sectionuniqueness}
Since the functional $\fl$ is not convex, in general, there is no reason to expect uniqueness of minimizers. Already in one dimension it is possible to construct a boundary datum $g$ giving two nonidentical minimizers. However, the cases with several minimizers are rare and we expect ``almost everywhere" a unique minimizer. We start with a general comparison principle, which can be applied to many different contexts.
\begin{lemma}
\label{orderedwrtdatum}
    Let $D$ be a bounded open domain of $\R^d$, and let $g,g'\in C(\overline{D})\cap H^1(D)$ with  $g'  \geq g \ge 0$ in $D$ and $g'(z)>g(z)>0$ at some $z$ in each connected component of $\partial D \cap \{g>0\}$. Then for corresponding minimizers to \eqref{minimizationproblem}, $u_g$ and $u_{g'}$, we have $u_{g'}\geq u_g$ on $\overline D$.
\end{lemma}
\begin{proof}
Since on $\{u_g =0\}$ the result holds trivially, consider the open set $\Omega_{u_g} = \{u_g >0\}\cap D$.

Define $\Tilde u(x) \coloneq \max \left\{ u_{g'}(x), u_g(x)\right\}$. 
Since by a computation (\cite[Lemma~2.5]{velichkov})
\begin{equation*}
    \fl(\max\{u_{g'},u_{g}\},D) + \fl(\min\{u_{g'},u_{g}\},D) = \fl(u_{g'},D) + \fl(u_g,D),
\end{equation*}
we have that $\Tilde u$ is also a minimizer with $\Delta \Tilde u = 0$ in $\Omega_{u_g} \subset \Omega_{\Tilde u}$. 
Suppose for contradiction that there exists $x_0 \in \Omega_{u_g}$ such that $u_g(x_0) > u_{g'}(x_0)$, that is $\Tilde u(x_0)= u_g(x_0)$. Let $\mathcal{C}$ be the connected component of $\Omega_{u_g}$ containing $x_0$. 

Set $h(x) \coloneq \Tilde u(x) -u_g(x) \geq 0$, then on $\mathcal{C}$, $\Delta h=0$ and as its minimum value of $0$ is attained at $x_0 \in \mathcal{C}$, by the strong maximum principle $h\equiv 0$ and $\Tilde u = u_g$ on $\overline{\mathcal{C}}$. 

We now show for the sake of contradiction that $\overline{\mathcal{C}}$ contains boundary points where $\Tilde u > u_g>0$.
If $\partial \mathcal{C} \subset \{u_g=0\}$, then by the maximum principle, $u\equiv 0$ in $\mathcal{C}$, contradicting $\mathcal{C} \subset \Omega_{u_g}$. Hence $\partial \mathcal{C} \cap \Omega_{u_g}  \neq \varnothing$. Let $x_1 \in \partial \mathcal{C} \cap \Omega_{u_g}$, if $x_1 \in \textnormal{int}(D)$, then by interior Lipschitz continuity $x_1$ cannot be a point in $\partial C$, i.e.\ $x_1 \in \partial D$ with $g(x_1)>0$.
Now the whole component $\partial D \cap \{g>0\}$ containing $x_1$ is contained in $\partial \mathcal{C}$ by continuity from Theorem~\ref{theoremcontinuityuptotheboundary}. 
But by assumption we have $z \in \bar C$ with
$$\Tilde u(z) \geq u_{g'}(z) = g'(z) > g(z) = u_{g}(z),$$
a contradiction to the fact that $\Tilde u = u_g$ on $\mathcal{C}$, finishing the proof.   
\end{proof}


We are now able to prove the generic uniqueness for the one-phase problem, using the argument from \cite[Proposition 1.2]{fernandezreal2023generic}, but for a wider class of boundary data.
\begin{proof}[Proof of Proposition~\ref{theoremgenericuniqueness}]
    During the proof we again use the Lipschitz continuity of minimizers, and we recall that we are taking $\Lambda=1$. By Lemma~\ref{orderedwrtdatum}, minimizers are ordered with respect to the boundary datum, i.e. $t' > t>0$ implies that $u_{t'} \geq u_{t}$. 

Let $t$ be such that there are at least two distinct minimizers $u_t^1, u_t^2$. Let $u_t^+ = \max\{u_t^1, u_t^2\}$ and $u_t^- = \min\{u_t^1, u_t^2\}$, then
\begin{equation*}
\begin{split}
    \fl(u_t^+) + \fl(u_t^-) &= \int_{D \cap \{u_t^1 \geq u_t^2\}} |\nabla u_t^1|^2 + \int_{D \cap \{u_t^1 < u_t^2\}} |\nabla u_t^2|^2 \\
    & \quad +  |(\{u_t^1>0\} \cup \{u_t^2>0\}) \cap D| \\
    &  \quad + \int_{D \cap \{u_t^1 \geq u_t^2\}} |\nabla u_t^2|^2 + \int_{D \cap \{u_t^1 < u_t^2\}} |\nabla u_t^1|^2 \\\
    & \quad + |\{u_t^1>0\} \cap \{u_t^2>0\} \cap D| \\
    &= \fl(u_t^1) + \fl(u_t^2).
\end{split}
\end{equation*}
As $u_t^1$ and $u_t^2$ are minimizers, so are $u_t^+$ and $u_t^-$. Now let $x_0$ be a point where $u_t^1$ and $u_t^2$ differ, i.e.\ without loss of generality $u_t^1(x_0) - u_t^2(x_0) = \eps$. By Lipschitz continuity, there exists $B_r(x_0)$ where $u_t^1(x_0) - u_t^2(x_0) > \eps/2$. Thus for $\rho = \min \{\eps/3, r\}$, there exists a $d+1$ dimensional ball $B_\rho$ such that $B_\rho \subset \textrm{epi}( u_t^2) \backslash \textrm{epi} (u_t^1)$. 

Repeating the argument for any $t$ with non-unique minimizers gives a collection of disjoint (as minimizers are ordered with respect to the boundary datum) balls. Since there can be at most countably many disjoint open balls in $\R^d$, the proof is finished.
\end{proof}

\subsection{Generic regularity of the free boundary}
\label{sectionregularity}
We now prove a slightly weaker version of Theorem~\ref{theoremgenericregularity notequicontinuous}, namely, the case  where the boundary data is equicontinuous.
The main part is already done in \cite[Section 4]{fernandezreal2023generic}, it remains only to prove \cite[Lemma 4.3]{fernandezreal2023generic} for the larger class of boundary data (i.e.\ equicontinuous, which then implies the result for continuous data, see Lemma \ref{countablesubfamilies}) instead of equi-Lipschitz). 

\begin{proposition}
\label{theoremgenericregularity}
    Theorem~\ref{theoremgenericregularity notequicontinuous} holds under the added assumption that the family $g_t$ is equicontinuous.
\end{proposition}

\begin{proof}
For simplicity of the exposition, we assume that $D=B_1$. 
Fix $\tau \in (0,1)$ and let $x_0$ be a free boundary point, $x_0 \in B_{1/2}\cap \partial \Omega_{u_{t_0}}$. for some $t_0 \in [\tau,1)$. In view of \cite[Lemma 4.3]{fernandezreal2023generic}, we need to show that there exists $\kappa=\kappa(\tau, d, \omega, M)>0$ (with $\omega$  the common modulus of continuity of $\{g_t\}$ and $M$ its uniform $H^1$ bound), such that
\begin{equation*}
    \sup_{B_{\kappa(t-t_0)}(x)} u_{t_0} \leq u_t(x) \qquad \text{for } x \in B_{3/4} \text{ and } t\in [t_0,1).
\end{equation*}
By the assumptions on $\{g_t\}$, we have 
\begin{equation*}
    g_{t_0} \geq \tau > 0 \qquad \text{ on } \partial B_1.
\end{equation*}
Applying now Theorem~\ref{theoremcontinuityuptotheboundary}
($u_{t_0}$ is continuous up to the boundary with modulus of continuity $\bar \omega (d, \omega, M)$), gives $\delta=\delta(\tau, d, \omega,M)<1/16$ such that
\begin{equation*}
    u_{t_0}> 0 \qquad \text{in } U:=B_1 \backslash \overline{B_{1-8 \delta}}.
\end{equation*}
Using now the comparison lemma, Lemma~\ref{orderedwrtdatum}, gives $u_t \geq u_{t_0}$ for $t\geq t_0$ and as $U \subset \Omega_{u_{t_0}} \subset \Omega_{u_t}$, 
\begin{equation*}
    \Delta (u_t-u_{t_0}) =0 \qquad \text{in }U.
\end{equation*}
The rest of the proof follows analogously to \cite[Lemma 4.3]{fernandezreal2023generic}.
\end{proof}

We set out to remove the assumption of equicontinuity in the previous statement. 
The rough idea is to partition a family of continuous (not necessarily equicontinuous) functions into countable subfamilies of equicontinuous functions, apply Theorem~\ref{theoremgenericregularity} on each subfamily and then show that the size of the set of functions not falling into any equicontinuous subfamily is small. This is due to the separability of continuous functions.

\begin{lemma}
\label{countablesubfamilies}
Let $d\geq 1$ and $D \subset \R^d$ a bounded Lipschitz domain.
    Let $ \{g_t\}_{t\in (0,1)}  $ be a monotone  family of continuous functions in $\overline{D}$, i.e.\ 
    $$ t \ge s \quad \implies\quad  g_t \geq g_s\quad\text{in}\quad \overline{D}.$$
    Then, there exists a countable family of disjoint open intervals $\{I_k\}_{k\in \mathbb{N}}$ with $I_k \subset (0, 1)$ such that
    \begin{itemize}
        \item $(0, 1)\setminus \bigcup_{k\in \mathbb{N}} I_k$ is countable,  
        \item  $g_t$ are locally equicontinuous for  $t\in I_k$. That is, for any $K\subset I_k$ compact, the family $\{g_t\}_{t\in K}$ is equicontinuous.   
    \end{itemize}
\end{lemma}

\begin{proof}
Let $T \subset (0,1)$ be the set where $\{g_t\}_{t\in (0, 1)}$ is locally equicontinuous, i.e.\
\begin{equation*}
    T:= \{ t\in (0,1): \quad \exists \eta>0 \text{ s.t. } \{g_s\}_{(t-\eta,t+\eta)} \text{ is equicontinuous} \}.
\end{equation*} 
Since for $t_0 \in T$ and $t\in B_{\eta/2}(t_0)$ (where $\eta=\eta(t_0)$), 
$$\{g_s\}_{(t-\eta/2,t+\eta/2)} \subset \{g_s\}_{(t_0-\eta,t_0+\eta)},$$
it follows that $T$ is open in $(\{g_s\}_{(0,1)}, \|\cdot\|_{L^\infty(D)})$.
Next, we show by contraposition that for a fixed $t\in T^c$, either 
\begin{equation*}
    \exists \eps>0 : \inf_{s>t} \|g_t-g_s\|_{L^\infty(D)} \geq \eps \quad \text{ or } \quad \exists \eps>0 : \inf_{s<t} \|g_t-g_s\|_{L^\infty(D)} \geq \eps.
\end{equation*}
Indeed, suppose it is not true true. Since $\{g_s\}_{(0,1)}$ is monotone, we have 
\begin{equation*}
\begin{split}
    \forall \eps>0,\, & \inf_{s>t} \|g_t-g_s\|_{L^\infty(D)} < \eps \quad \text{ and } \quad \forall \eps>0 ,\, \inf_{s<t} \|g_t-g_s\|_{L^\infty(D)} < \eps \\
\end{split}
\end{equation*}
that is, 
\[
\lim_{s\to t} \|g_t-g_s\|_{L^\infty(D)} = 0.
\]
Let $\eps>0$ and take $\eta, \delta>0$ such that 
\begin{equation*}
    \begin{split}
    |s-t| < \eta &\implies \|g_t-g_s\|_{L^\infty(D)} < \frac{\eps}{3},\\
    |x-y| < \delta &\implies |g_t(x)-g_t(y)| < \frac{\eps}{3}.
    \end{split}
\end{equation*}
By the triangle inequality,
\begin{equation*}
    |g_s(x)-g_s(y)| \leq |g_s(x)-g_t(x)| + |g_t(x)-g_t(y)| + |g_t(y)-g_s(y)| \leq \eps,
\end{equation*}
it follows that $\{g_s\}_{(t-\eta,t+\eta)}$ is equicontinuous, as $\delta$ does not depend on $s$, giving a contradiction. 

It remains to show that $T^c$ is countable. We have
\[
T^c = \underbrace{ \{t \in (0,1): \inf_{s>t} \|g_t-g_s\|_{L^\infty(D)} >0 \}}_{=:T^c_1} \cup  \underbrace{\{t \in (0,1): \inf_{s<t} \|g_t-g_s\|_{L^\infty(D)} >0 \}}_{=:T^c_2}.
\]
 Let $$ J_m := \{t \in (0,1): \inf_{s>t} \|g_t-g_s\|_{L^\infty(D)} \geq \tfrac{1}{m}\}, \quad\text{for}\quad m \in \mathbb{N}.$$
 Then, by definition, we have 
 \[
 T_1^c = \bigcup_{m\in \mathbb{N}} J_m. 
 \]

For any $m\in \mathbb{N}$ and $t_1, t_2\in J_{m}$ with $t_2 \neq t_1$ we have $\|g_{t_1} - g_{t_2}\|_{L^\infty(D)} \ge \tfrac{1}{m} > 0$, and so the family $\{g_t\}_{t\in J_{m}}$ is a family of continuous functions that are pairwise at distance $\tfrac{1}{m}$.
However, $C(\bar D)$ and then also $\{g_t\}_{t \in J_{m}}$ are separable, that is, there exists a countable subset $I \subset J_{m}$ such that $\{g_t\}_{t \in I}$ is dense in $\{g_t\}_{t \in J_{m}}$. Take $g_t \in \{g_t\}_{t \in J_{m}}$ and $t_n \in I \subset J_{m}$ with $\|g_{t_n} - g_t\|_{L^\infty(D)} \to 0$. The lower bound on the pairwise distance implies that $t_n = t$ for all $n$ large enough, i.e. for any $t\in J_{m}$, $g_t \in \{g_t\}_{t \in I}$, thus $J_m$ is countable. Since $m$ was arbitrary and countability is preserved under countable unions, $T^c_1$ is countable and by the same argument $T^c_2$ is countable as well, and so is $T^c$.

Since $T \subset (0,1)$ is open it can be written as a disjoint union of open intervals, that is $(0,1) = T^c \cup T = T^c \cup \bigcup_{k \in \mathbb{N}} I_k$. Now,  $K \subset \subset I_k$, $\{(t-\eta_t, t+\eta_t\}_{t \in K}$ admits a finite subcover; and we take as common modulus of continuity its maximum. 
\end{proof}

We now combine Proposition~\ref{theoremgenericregularity} with Lemma~\ref{countablesubfamilies} to obtain the generic regularity result for a general (not necessarily equicontinuous) family of continuous boundary data, thus proving Theorem \ref{theoremgenericregularity notequicontinuous} in its full generality.

\begin{proof}[Proof of Theorem~\ref{theoremgenericregularity notequicontinuous}]
    We first treat the case $d=d^* +1$. By Lemma~\ref{countablesubfamilies}, let $J_0= (0, 1)\setminus \bigcup_{k\in \mathbb{N}} I_k$ and $\{g_s\}_{s}$ locally equicontinuous for $s\in I_k$ for each $k$. By taking a countable compact exhaustion of $I_k$ and applying Theorem~\ref{theoremgenericregularity} in each compact, we deduce that, for each $k$,  there is $J_k$ countable such that $\textnormal{Sing}(u_t) = \varnothing$ on $I_k \backslash J_k$. The result follows by setting $J:= J_0 \cup \bigcup_{k \in \mathbb{N}} J_k$.
    
    For the case $d\geq d^* +2$, since the dimension estimate $\dim_{\Hh} \textnormal{Sing}(u_t) \leq d-d^*-2$ holds for almost every $t \in K$ and every $K \subset \subset I_k$, it holds a.e.\ on $I_k$. Hence it holds also almost everywhere on $(0,1)$. 
\end{proof}

\printbibliography 

\end{document}